\documentclass[graphicx]{article}
\usepackage{amsthm,amsfonts, amsbsy, amssymb,amsmath,graphicx}
\usepackage{graphics}
\usepackage[english]{babel}
\newtheorem{thm}{Theorem}
\newtheorem{lem}{Lemma}

\newtheorem{crl}{Corollary}
\newtheorem{rk}{Remark}

\theoremstyle{definition}
\newtheorem{dfn}{Definition}

\def\:{\colon}

\def\R{{\mathbb R}}
\def\Z{{\mathbb Z}}
\def\0{{\mathbf 0}}
\def\1{{\mathbf 1}}

\def\R{{\mathbb R}}

\title{A One-term Parity Bracket for Braids}

\author{Vassily Olegovich Manturov\footnote{The author is partially
supported by Laboratory of Quantum Topology of Chelyabinsk State
University (Russian Federation government grant 14.Z50.31.0020), by
RF President NSh — 1410.2012.1,
 and
by grants of the Russian Foundation for Basic Resarch,
13-01-00830,14-01-91161, 14-01-31288.}}

\date{}

 \def\R{{\mathbb R}}
 \def\Z{{\mathbb Z}}

\begin{document}

\title{One-Term Parity Bracket For Braids}

\maketitle

\begin{abstract}In previous papers (see e.g., \cite{Parity}), the author realized
the following principle for many knot theories: {\em if a knot
diagram is complicated enough then it reproduces itself, i.e., is a
subdiagram of any other diagram equivalent to it}. This principle is
realized by diagram-valued invariants $[\cdot]$ of knots such that
$[K]=K$ for $K$ complicated enough.

It turns out that in the case of free braids, the same principle can
be realized in an unexpectedly easy way by a one-term invariant
formula.

\end{abstract}


{\em To Lou Kauffman on the occasion of his 70-th birthday}

\section{Introduction}

Free knots \cite{Parity} (and later, free braids, \cite{IMN})
appeared as the simplest natural simplification of virtual knots
(and braids): at each virtual crossing we keep only the information
that this crossing exists, and the moves are natural general
position intersection moves as they appear for curves in
$2$-surfaces: the three Reidemeister moves for knots and the
braid-like second and third Reidemeister moves for braids.

This objects (without over/undercrossings) turned out to be
extremely non-trivial and carrying very important information about
virtual knots/braids and other objects which could be represented by
curves with generic intersections.

 Assume topological {\em objects} (knots, braids,
etc.) are encoded by {\em diagrams} (words) modulo {\em moves}
(relations). It turns out that in many situations {\em if an object
is complicated enough} then it appears as a sub-object of every
object equivalent to it.

In \cite{Parity}, the author introduced the study of {\em parity}
into knot theory; the parity is a sophisticated way of
distinguishing between {\em even} and {\em odd} nodes (crossings,
letters) which behave nicely under {\em moves} (relations).

In \cite{Parity}, the above principle  is first realized for {\em
free knots} where ``complicated enough'' means {\em irreducible} (in
some natural sense) and odd (with all nodes being odd).

In the present paper, we give a similar but much simpler
construction for {\em free braids.} Free braids are a simplification
of virtual braids; without giving a definition of virtual braids, we
say that virtual braids have a natural homomorphism onto the set of
free braids with the same number of strands.

Denote by ${\cal F}_{n}$ the group generated by $\zeta_{1},\dots,
\zeta_{n-1},\tau_{1},\dots, \tau_{n-1}$ subject to the following
relations:

\begin{enumerate}

\item (Second Reidemeister move) $\tau_{i}^{2}=1,\zeta_{i}^{2}=1, i=1,\dots, (n-1)$;

\item (Virtualization) $\tau_{i}\zeta_{i}=\zeta_{i}\tau_{i}, i=1,\dots,
(n-1)$;

\item (Far commutativity)
$\zeta_{i}\zeta_{j}=\zeta_{j}\zeta_{i};$\;$\zeta_{i}\tau_{j}=\tau_{j}\zeta_{i}$,$\tau_{i}\tau_{j}=\tau_{j}\tau_{i},
i,j=1,\dots, (n-1),|i-j|\ge 2$;

\item (Virtual third Reidemeister move): $$\tau_{i}\tau_{i+1}\tau_{i}=\tau_{i+1}\tau_{i}\tau_{i+1}, i=1,\dots, n-2;$$

\item (Semivirtual third Reidemeister move): $$\tau_{i}\tau_{i+1}\zeta_{i}=\zeta_{i+1}\tau_{i}\tau_{i+1},i=1,\dots, (n-2).$$

\end{enumerate}

Here $\zeta_{i},i=1,\dots, n-1$ are called {\em classical
generators}; $\tau_{i},i=1,\dots, (n-1)$ are called {\em virtual
generators}. Respectively, the second Reidemeister move which deals
with $\zeta_{i}$ is called {\em classical} and the one which deals
with $\tau_{i}$, is called {\em virtual}.

\begin{dfn}
The {\em free $n$-strand braid group} ${\cal FB}_{n}$ is the
quotient group of the group ${\cal F}_{n}$ modulo {\em the third
Reidemeister moves}:
$$\zeta_{i}\zeta_{i+1}\zeta_{i}=\zeta_{i+1}\zeta_{i}\zeta_{i+1},i=1,\dots, n-2.$$
\end{dfn}

A word in $\zeta_{1},\dots, \zeta_{n-1},\tau_{1},\dots, \tau_{n-1}$
will be called an {\em $n$-strand free braid-word} or, for brevity,
just {\em braid word} or {\em word} when the number of strands is
clear from the context.

A {\em free $n$-strand braid} is an element of ${\cal FB}_{n}$.

Analogously, by a {\em cyclic braid-word} we mean a free braid-word
considered up to the cyclic permutation of letters. A {\em cyclic
$n$-strand free braid} is a conjugacy class of the group ${\cal
FB}_{n}$.

\begin{dfn}
Having a braid $\beta$, we denote the corresponding cyclic braid by
$cl(\beta)$ and call it the {\em closure} of $\beta$.
\end{dfn}

With each generator $\zeta_{i}$ or $\tau_{i}$ we associate a diagram
in $\R^{1}_{x}\times [0,1]_{y}$ consisting of $n-2$ vertical lines
connecting points $(j,0)$ to $(j,1)$, $j\neq i,j\neq i+1$ and two
straight lines connecting $(i,0)$ to $(i+1,1)$ and $(i,1)$ to
$(i+1,0)$. The intersection point is encircled in the case of the
virtual generator $\tau_{i}$ and is marked by a solid dot for the
classical generator $\zeta_{i}$. Every free braid-word $\beta$ in
$\zeta_{i},\tau_{i}$ can be depicted by a diagram on $n$ strands by
reading it from the top to the bottom, juxtaposing and rescaling the
pictures corresponding to generators of the braid-word. Thus, having
a word in $k$ letters, we get $k$ crossings in the layers
$\frac{k-1}{k}\le y\le 1, \frac{k-2}{k}\le y\le \frac{k-1}{k},\dots,
0\le y\le \frac{1}{k}$.

Thus, each diagram of a free braid consists of {\em strands} passing
through {\em crossings}: there are $n$ strands starting from
$(1,1),\dots, (n,1)$ and going downwards; for each generator
$\sigma_{i}$ or $\zeta_{i}$, some two strands intersect at this
crossing. Note that the numbers of these two strands passing through
a crossing $\zeta_{i}$ in a braid-word $\beta$ can be arbitrary
since they are counted not locally but according to their endpoints
for $y=1$.

Analogously, for cyclic braids, we can define a diagram not in
$\R^{2}$ but in $\R^{1}_{x}\times S^{1}_{y}$, where $S^{1}$ is the
circle obtained by identifying the ends of the interval.

This naturally defines the {\em permutation} $P(\beta)$ of the
braid-word $\beta$: if the braid connects the upper end $(k,1)$ to
$(f(k),0)$, then $P(\beta)$ takes $k$ to $f(k)$ for $k=1,\dots,
n-1$.

When we pass from a braid word $\beta$ to its closure $cl(\beta)$,
{\em strands} connect to each other and close up to some circles:
the number of circles is equal to the number of cycles of
$P(\beta)$. In particular, if $P(\beta)$ is cyclic, we have exactly
one cycle.

\begin{figure}
\centering\includegraphics[width=200pt]{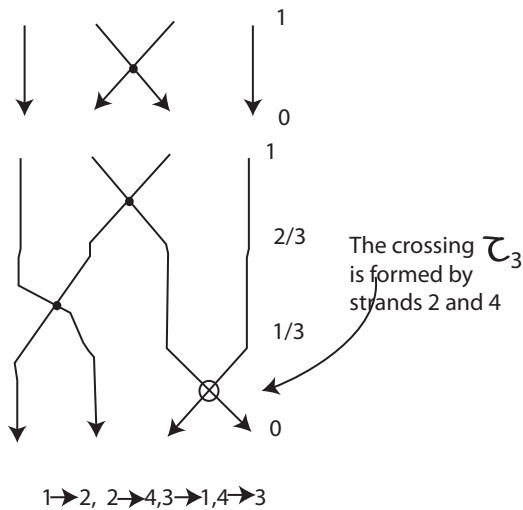} \caption{A
braid generator; a braid; its permutation} \label{braidmoves}
\end{figure}

\begin{dfn}
A {\em chord diagram} is a $3$-regular graph with a selected
oriented cycle which passes through all vertices; this cycle is
called the {\em core} of the chord diagram; the remaining edges are
called {\em chords} of the chord diagram; chords are not oriented.

Two chord diagrams are considered up to a homeomorphism of the core
circle which takes core circle to core circle and preserves the
orientation.
\end{dfn}

In the case of free braids with cyclic permutation we can define the
{\em chord diagram} $C(cl(\beta))$ as follows. The whole diagram
$cl(\beta)$ can be considered as the image of map $f:S^{1}_{\phi}
\to cl(\beta)$; the circle is oriented according to the orientation
of strands (from the top to the bottom) having classical and virtual
crossings, so one segment of the circle, say $[0,\frac{1}{n}]$ is
mapped to the first strand from $y=1$ to $y=0$, the second segment
(say, from $[\frac{1}{n},\frac{2}{n}]$) is mapped to the strand
which connects the end of the first strand to the beginning of the
second strand, etc.
 This map is bijective outside preimages of crossings. For each
classical crossing $x$, we have exactly two preimages
$x_{1},x_{2}\in S^{1}$. Thus, we take all classical crossings and
connect the corresponding pairs of points by chords; this leads us
to a chord diagram. Certainly, the parametrization change of the
circle $S^{1}$ does not change the equivalence class of the
resulting chord diagram.

Note that we disregard virtual crossings when constructing chord
diagrams.

\begin{dfn}
We say that two chords $c,d$ of a chord diagram $D$ are {\em linked}
if two ends of $d$ belong to different components of the complement
$C\backslash c$ where $C$ is the core circle of $D$.
\end{dfn}

In the sequel, we shall need {\em permutation braids}. Namely, with
each permutation $P:(1\to p(1),\dots n\to p(n))$, one can associate
a braid diagram connecting $(k,1)$ with $(p(k),0)$ and having only
virtual crossings $\tau_{i}$. It follows from the definition that
for a fixed $P$, all such braid diagrams are equivalent by moves
which deal with virtual crossings only. Denote this braid by
$\beta_{P}$.

\begin{rk}
Note that the relation
$\tau_{i+1}\tau_{i}\zeta_{i+1}=\zeta_{i}\tau_{i+1}\tau_{i}$ is not
in the list because it can be expressed in terms of the relations
for ${\cal F}_{n}$.
\end{rk}

So, the relations for ${\cal F}_{n}$ or ${\cal FB}_{n}$ admit a
geometrical interpretation in terms of moves.

We shall often say {\em crossing} instead of {\em letter
(generator)} when it does not cause any confusion.


\begin{dfn}
A free braid diagram is {\em pure} if its permutation is the
identity; a free braid is {\em pure} if some (hence, all) braid
diagrams representing it are pure.
\end{dfn}

\begin{dfn}
A {\em cyclic free braid} is a conjugacy class of free braids.
\end{dfn}

One can naturally interpret closures of braids as diagrams of {\em
free knots} or {\em free links}: each closed strand gives rise to a
knot (link) component; however, to define free knots (links) one
needs additional moves which do not originate from braids.

Let $B$ be some class (set) of free braids. For example, we can take
$B$ all {\em pure braids} or all braids having permutations from
some fixed set.
 By a {\em parity} for braids from $B$
 we mean a way of associating
elements from $\Z_{2}$ with all classical crossings of all
braid-words $\beta$ representing braids from $B$ such that:

\begin{enumerate}

\item If two braid words $ApB\to AqB$ are obtained from one another by applying a
defining relation $p\to q$ from the list for ${\cal FB}_{n}$, then
parities of all crossings not taking part in these relations (i.e.,
crossings belonging to $A$ or to $B$) do not change.

\item When applying $p=\zeta_{i}\zeta_{j} \to q=\zeta_{j}\zeta_{i}, |i-j|\ge 2$,  the parity of the crossing $\zeta_{i}$
on the left hand side coincides with the parity on the right hand
side. The same holds for $\zeta_{j}$;

\item Analogously, for $p=\zeta_{i}\tau_{j}\to q=\tau_{j}\zeta_{i}, |i-j|\ge 2$ the
parity of $\zeta_{i}$ does not change;

\item For the classical second Reidemeister move $p=\zeta_{i}^{2}\to
1=q$ both $\zeta_{i}$ on the LHS are of the same parity;

\item For the relation
$p=\zeta_{i}\zeta_{i+1}\zeta_{i}\to\zeta_{i+1}\zeta_{i}\zeta_{i+1}=q$
we require that:

\begin{enumerate}

\item The number of odd crossings among
$\zeta_{i},\zeta_{i+1},\zeta_{i}$ on the LHS is even.

\item The parity of the upper $\zeta_{i}$ on the LHS coincides
with the parity of the lower $\zeta_{i+1}$ on the RHS;

\item The parity of the middle $\zeta_{i+1}$ on the LHS coincides
with that of the middle $\zeta_{i}$ on the RHS;

\item The parity of the lower $\zeta_{i}$ on the LHS coincides with
that of the upper $\zeta_{i+1}$ on the RHS.

\end{enumerate}

\item For the relation $p=\tau_{i}\tau_{i+1}\zeta_{i} \to q=\zeta_{i+1}\tau_{i}\tau_{i+1},$
the parity of $\zeta_{i}$ on the LHS coincides with the parity of
$\zeta_{i+1}$ on the right hand side.

\item For the virtualization relation
$\zeta_{i}\tau_{i}=\tau_{i}\zeta_{i}$, the parity of $\zeta_{i}$
does not change.

\end{enumerate}

\begin{rk}
Note that in \cite{Parity} and subsequent papers, the diagrams do
not take into account virtual crossings, and parities are defined by
using classical crossings only.
\end{rk}

Let us now define some parities. For all braids, one can define the
{\em component-wise} parities as follows. Let us split the set
$N=\{1,\dots, n\}$ of indices into two disjoint subsets
$N=N_{1}\sqcup N_{2}$. Now, every crossing formed by two strands
from the same subset $N_{i}$ is {\em even}. Every crossing formed by
two strands from different subsets $N_{1}$ and $N_{2}$ is {\em odd.}

Now, fix two permutations $P$ and $Q$ such that $P\circ Q$ is a
cyclic permutation. Then, for all braids having permutation $P$, we
define the {\em $Q$-Gaussian} parity as follows.

Let $\beta$ be a braid with permutation $P$. Consider the product
$\beta\cdot \beta_{Q}$ where $\beta_{Q}$ is the permutation braid
corresponding to $Q$. The resulting braid $\beta\cdot\beta_{Q}$ is
cyclic, thus, $cl(\beta\cdot \beta_{Q})$ has one strand.

We get a chord diagram $C(cl(\beta\cdot\beta_{Q}))$, where chords
correspond to classical crossings of $\beta$. We say that a
classical crossing of $\beta$ is even if the corresponding chord is
linked with evenly many chords.

The proof of the fact that these parities satisfy all parity axioms
are a slight modification of a similar proof for free links from
\cite{Parity}. They are left to the reader as exercises.

\section{The Main Invariant}

\begin{dfn} Let $p$ be a parity. Let the {\em one-term parity
bracket} for an $n$-strand braid word $\beta$ be the $n$-strand
braid word $[\beta]_{p}$ obtained from $\beta$ by removing all even
letters $\zeta_{i}$.
\end{dfn}

\begin{thm}
The map $\beta \to [\beta]_{p}$ is a well defined map from the set
of free braids (for which $p$ is defined) to ${\cal F}_{n}$; in
other words, if $\beta$ and $\beta'$ are equal as elements of ${\cal
FB}_{n}$, then $[\beta]_{p}$ and $[\beta']_{p}$ are equal as
elements of ${\cal F}_{n}$.
\end{thm}

\begin{proof}
Assume $\beta_{1}=Ar_{1}B, \beta_{2}=Ar_{2}B,$ where $r_{1}\to
r_{2}$ is some relation for ${\cal FB}_{n}$.

Then $[\beta_{1}]_{p}= {\tilde A}{\tilde r_{1}}{\tilde B},
[\beta_{2}]_{p}={\tilde A}{\tilde r_{2}}{\tilde B},$ where ${\tilde
A}$ and ${\tilde B}$ are obtained from $A$ and $B$ by removing even
letters; the rule for defining even or odd letters is the same for
$\beta_{1}$ and $\beta_{2}$.

Thus, it remains to show that ${\tilde r_{1}}$ and ${\tilde r_{2}}$
are equivalent as elements from ${\cal F}_{n}$.

Indeed, let us consider the relations from ${\cal F}_{n}$.

The far commutativity relations $r_{1}\to r_{2}$ yield either far
commutativity or the identity depending on the parity of crossings.
The virtualization move yields either virtualization or the
identity. The moves $\tau_{i}^{2}=1$ always yield $\tau_{i}^{2}=1$;
the move $\zeta_{i}^{2}=1$ yields either the identity or
$\zeta_{i}^{2}=1$ depending on the parity; the third virtual
Reidemeister move always yields the third virtual Reidemeister move.
The third semivirtual Reidemeister move yields either the identity
or the third semivirtual Reidemeister move.

Finally, the third classical Reidemeister move for three even
crossings leads to the identity since the words ${\tilde r_{1}}$ and
${\tilde r_{2}}$ are both empty.

Now, if for  $r_{1}=\zeta_{i}\zeta_{i+1}\zeta_{i}$ the last letter
$\zeta_{i}$ is even, and the other two letters are odd, we see that
${\tilde r_{1}}=\zeta_{i}\zeta_{i+1}={\tilde r_{2}}$; if the first
$\zeta_{i}$ in $r_{1}$ is even and the other two letters are odd,
then ${\tilde r_{1}}=\zeta_{i+1}\zeta_{i}={\tilde r_{2}}$; finally,
if $\zeta_{i+1}$ in $r_{1}$ is even and both $\zeta_{i}$ are odd, we
see that ${\tilde r_{1}}=\zeta_{i}^{2}$ and ${\tilde
r_{2}}=\zeta_{i+1}^{2}$; these two words are equivalent by the
second Reidemeister classical moves.

Now, if $\beta$ and $\beta'$ are equivalent as elements of ${\cal
FB}_{n}$ then this equivalence can be represented as a sequence
$\beta'=\beta_{1}\to \beta_{2}\to\cdots\to \beta_{l}=\beta'$ where
each two neighbouring $\beta_{i}$ and $\beta_{j}$ are related as
described above; thus, $[\beta]_{p}=[\beta']_{p}.$
\end{proof}

The following important fact follows from the definition.
\begin{crl}
Let $p$ be a parity. Let $\beta$ be a free $n$-strand braid-word
with all odd crossings with respect to $p$. Then

$$[\beta]_{p}=\beta.$$

Here on the left hand side, $\beta$ is considered as an element of
${\cal FB}_{n}$, and on the right hand side $\beta$ is an element of
${\cal F}_{n}$. \label{maincrl}
\end{crl}

It turns out that the word problem for ${\cal F}_{n}$ is extremely
easy to solve.

\begin{dfn}
We say that two braid-words $\beta_{1}$ and $\beta_{2}$ are {\em
strongly equivalent} if they are equivalent by all moves from ${\cal
F}_{n}$ except the second classical Reidemeister moves
$\zeta_{i}^{2}=1$.
\end{dfn}

Every braid-word $b$ can be thought of as an immersion of a graph in
$\R^{2}$. This graph $\Gamma(b)$ has $2n$ vertices corresponding to
endpoints of $b$, and four-valent vertices corresponding to all {\em
classical} crossings of $b$. Virtual crossings are not vertices of
the graph; they just lie on edges of $\Gamma(b)$. Besides,
$\Gamma(b)$ is endowed with an additional information. All upper and
lower vertices are enumerated; all edges are oriented downwards.
Besides these ordering of final points and orientation of edges,
this graph also possesses the ordering: for each crossing we
indicate which edge coming to this crossing is opposite to which
edge emanating from this crossing downwards.

\begin{lem}
Two braid-words $\beta,\beta'$ are strongly equivalent if and only
if $\Gamma(b)$ is equivalent to $\Gamma(b')$ with all structures
(orientation, ordered upper vertices ordered lower vertices opposite
edges) preserved.
\end{lem}

\begin{dfn}
Let $\beta$ be an $n$-strand braid-word. Let $x,x'$ be some two
classical crossings of a braid-word $\beta$ lying on the same
strands of $\beta$ (say, number $i$ and number $j$). We say that
$x,x'$ {\em form a bigon} if in $\beta$ there is no classical
crossing letter $\zeta_{k}$ between the two letters corresponding to
$x$ and to $x'$ and belonging to either $i$-th or $j$-th strand.

By the {\em bigon reduction} we mean the operation which deletes
$x,x'$ from $\beta$.

If $\beta'$ can be obtained from $\beta$ by a sequence of bigon
reductions, we say that $\beta'$ is a {\em descendant} of $\beta$
and write $\beta\to \beta'$.
\end{dfn}

It can be easily shown that the resulting braid $\beta'$ is
equivalent to $\beta$ in ${\cal F}_{n}$.

Let $\beta_{1}$ and $\beta_{2}$ be two strongly equivalent
braid-words. We have a bijection $u$ between the set of their
classical crossings. This bijection comes from the isomorphism
between graphs $\Gamma(\beta_{1})$ and $\Gamma(\beta_{2})$. All
bigons in the initial braids correspond to bigons in these graphs.
This obviously leads to the following

\begin{lem}
If two crossings $x_{1}$ and $x_{2}$ form a bigon, then $u(x_{1})$
and $u(x_{2})$ form a bigon, and the braid-words $\beta'_{1}$ and
$\beta'_{2}$ resulting from these bigon reductions are pairwise
strongly equivalent.
\end{lem}

\begin{dfn}
We say that a braid-word $\beta$ in ${\cal F}_{n}$ is {\em
irreducible} if it admits no bigon reduction.
\end{dfn}

Note that the second classical Reidemeister move is a partial case
of the bigon reduction.

\begin{lem}
Assume two classical crossings $x,x'$ of $\beta$ form a bigon with
the bigon reduction $\beta\to \beta'$ and $x,x''$ form a bigon of
$\beta$ with the bigon reduction $\beta\to \beta''$. Then the
resulting braid-words $\beta'$ and $\beta''$ are strongly
equivalent.
\end{lem}

\begin{proof}
Indeed, it suffices to look at the graph $\Gamma(\beta)$ and see the
three vertices in a sequence of two bigons. The result of bigon
reduction leads to isomorphic graphs.
\end{proof}

\begin{thm}
Every element $b$ of ${\cal F}_{n}$ has an irreducible braid-word
$\beta_{0}$ representing it; all irreducible braid-words
representing $b$ are strongly equivalent.
\end{thm}

\begin{proof}
Start with any braid $\beta$ representing $b$ and apply bigon
reductions when possible; when we get an irreducible representative,
denote it by $\beta_{0}$.

We want to prove that all irreducible descendants of every
braid-word are strongly equivalent. Assume there is a counterexample
$\gamma$ which is minimal with respect to the number of classical
crossings.

Assume $\gamma$ has only one bigon and admits only one bigon
reduction $\gamma\to \gamma'$; then all irreducible descendants of
$\gamma$ are irreducible descendants of $\gamma'$. Thus, $\gamma'$
has different descendants and hence $\gamma$ is not minimal.

Now, we assume that there are bigon reductions $\gamma\to \gamma'$
and $\gamma\to \gamma''$ such that $\gamma'$ and $\gamma''$ have
irreducible descendants which are not strongly equivalent. If the
bigons for these two reductions share a vertex then $\gamma'$ and
$\gamma''$ are strongly equivalent, so, all their irreducible
descendants are strongly equivalent.

Now, if the bigon reduction $\gamma\to \gamma'$ is performed at two
crossings $p,q$ and the bigon reduction $\gamma\to \gamma''$ is
performed at two crossings $r,s$ where all crossings $p,q,r,s$ are
distinct, then $\gamma'$ and $\gamma''$ have a common descendant
$\gamma'''$ obtained from $\gamma$ by deleting letters $p,q,r,s$.
Now, all descendants from $\gamma'$ are strongly equivalent to each
other, thus, they are strongly equivalent to all descendants of
$\gamma'''$, and the latter are all strongly equivalent to all
descendants of $\gamma''$. The contradiction completes the proof.

\end{proof}

Thus, Corollary \ref{maincrl} realizes the main principle formulated
in the very beginning of the paper. Namely, if we identify free
braid diagrams which are strongly equivalent, then Theorem
\ref{maincrl} can be reformulated as

\begin{thm}
Let $p$ be a parity for (some class of) free braids. Let $\beta$ be
a free braid for which $p$ is defined. If all crossings of $\beta$
are odd and no bigon reduction can be applied to a braid-word
$\beta$ then every other braid-word $\beta'$ equivalent to it in
${\cal F}_{n}$ contains a subword which is strongly equivalent to
$\beta$.
\end{thm}

\begin{proof}
Indeed, $[\beta']_{p}=[\beta]_{p}=\beta$. Recalling that
$[\beta']_{p}$ is obtained from $\beta'$ by removing some crossings,
and taking into account that $\beta$ is irreducible, we see that
$\beta$ is strongly equivalent to some subword of $[\beta']_{p}$,
hence, $\beta$ is strongly equivalent to a subword of $\beta'.$
\end{proof}

\begin{rk}
Actually, with some more elaborated techniques (e.g., along the
lines of \cite{KM}), one can prove the same theorem for weaker
condition on crossings on $\beta$. We shall touch on this as well as
on a complete algorithmic recognition of free braids in a subsequent
paper.
\end{rk}

Thus, by looking at $[\beta]_{p}$ we can judge about all possible
words equivalent to $\beta$.

\section{A Corollary}

The invariance of the parity bracket has one important corollary.
For oriented classical, virtual, and free knots there are
principally different types of the second and the third classical
Reidemeister moves.

The second and the third moves which originate from braids look as
shown in Fig.\ref{orient}.

\begin{figure}
\centering\includegraphics[width=200pt]{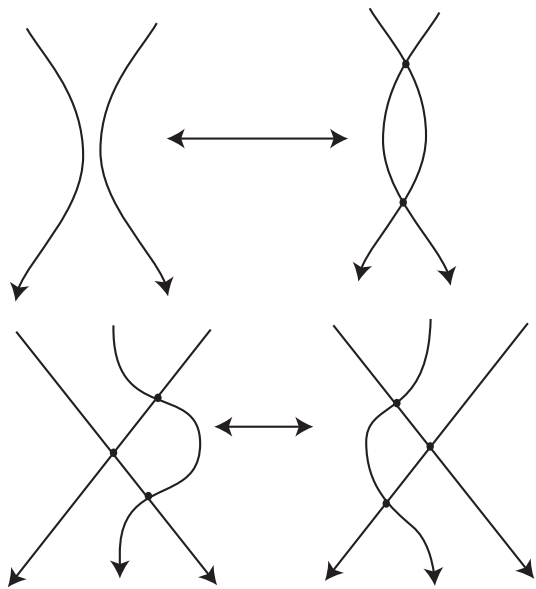}
\caption{Braid-Like Reidemeister Moves} \label{orient}
\end{figure}

Besides them, there are unoriented second and third Reidemeister
moves shown in Fig. \ref{noncoherent}.

\begin{figure}
\centering\includegraphics[width=150pt]{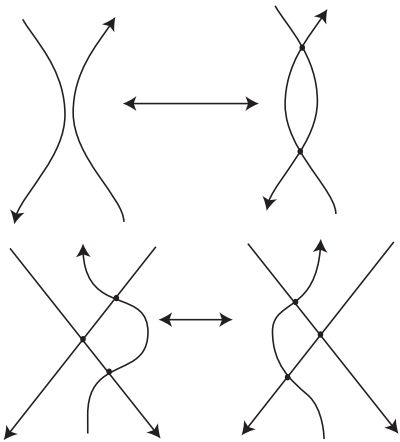}\caption{Unoriented
Reidemeister moves}\label{noncoherent}
\end{figure}

The classical Markov theorem says that closures of two classical
braids  $\beta_{1},\beta_{2}$ yield equivalent links if and only if
$\beta_{2}$ can be obtained from $\beta_{1}$ by  braid moves and the
stabilization move (and its inverse). The stabilization move for an
$n$-strand braid adds one new strand and a crossing between this new
strand and its neighbouring strand.

On the level of diagrams, braid moves are the second and third
Reidemeister moves, and the stabilization move (Markov move) is the
first Reidemeister move. Thus, we can use only braid-like second and
third moves together with the first Reidemeister move.

In the case of free  braids, we have virtual second Reidemeister
moves, virtualizations, far commutativity, virtual and semivirtual
moves. These moves are not interesting because they do not change
the underlying graph and the strong equivalence class.

As for those moves which do change the strong equivalence class, we
have classical second Reidemeister move and classical third
Reidemeister move.

For free knots (as well as for virtual knots and their analogues),
all Reidemeister moves contain unoriented Reidemeister moves as
well.

Unlike the classical case, Markov's theorem for virtual knots and
free knots (see \cite{LR}, \cite{Kamada}, and \cite{MW}) require
some unoriented versions of the second and the third Reidemeister
moves.

Without giving detailed definitions and going into details, we
formulate the following

\begin{thm}
Unoriented Reidemeister moves for free (flat,virtual) links can not
be expressed in terms of braid-like Reidemeister moves, the first
Reidemeister move, the detour move.
\end{thm}

Indeed, one can define the one-term bracket for Gaussian parity for
free knots in a way similar to braids. This bracket is invariant
under braid-like Reidemeister moves and adds one extra component
under the first Reidemeister move.

However, the bracket changes crucially when we perform an unoriented
second Reidemeister move with two even classical crossings.

See Figure \ref{brackRM}.

\begin{figure}
\centering\includegraphics[width=150pt]{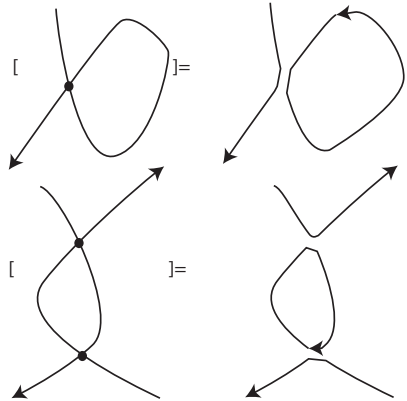} \caption{The
behaviour of the one-term bracket for free knots} \label{brackRM}
\end{figure}

The definition and the invariance proof for braid-like moves are
essentially the same as for the case of braids.

Let $\beta$ be the ``brunnian'' free $n$-strand braid, see
Fig.\ref{brbr}.

\begin{figure}
\centering\includegraphics[width=200pt]{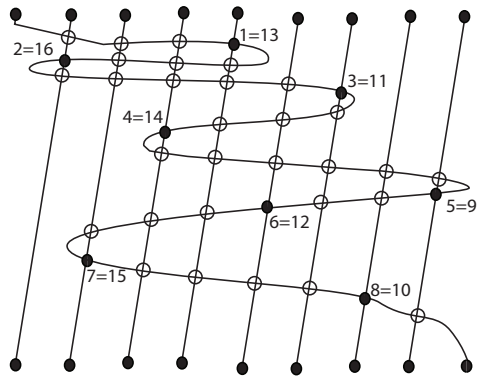}\caption{The
Brunnian Braid $\beta$} \label{brbr}
\end{figure}

The corresponding word is

$$\tau_{1}\tau_{2}\tau_{3}\zeta_{4}\tau_{4}\tau_{3}\tau_{2}\zeta_{1}\tau_{1}\tau_{2}\tau_{3}
\tau_{4}\tau_{5}\zeta_{6}\tau_{6}\tau_{5}\tau_{4}\zeta_{3}\tau_{3}\tau_{4}\tau_{5}\tau_{6}\tau_{7}$$

$$\times
\tau_{8}\zeta_{8}\tau_{7}\tau_{6}\zeta_{5}\tau_{4}\tau_{3}\tau_{2}\zeta_{2}\tau_{3}\tau_{4}\tau_{5}
\tau_{6}\zeta_{7}\tau_{8}.$$

Its permutation is cyclic; let us consider the Gaussian parity $p$
for its closure. One can easily see that all crossings of $\beta$
are odd. Indeed, if we start walking from the upper end of the first
strand, we meet each of the strands $2,\dots, 9$ once; the order of
crossings (each counted twice) is shown in Fig. \ref{brunn}.

Thus, when taking $[\beta]_{p}=[\beta]$ for the Gauusian parity $p$,
and the closure $Cl(\beta)$ is odd and admits no bigon reduction.

Thus, we will have exactly one term in the bracket for the
corresponding free knot.

Now, let us transform the braid by adding a new strand and two new
crossings, \ref{brunn}. This braid is again cyclic (the two new ends
appeared in the left, and the two new crossings are in the bottom
left).

\begin{figure}
\centering\includegraphics[width=200pt]{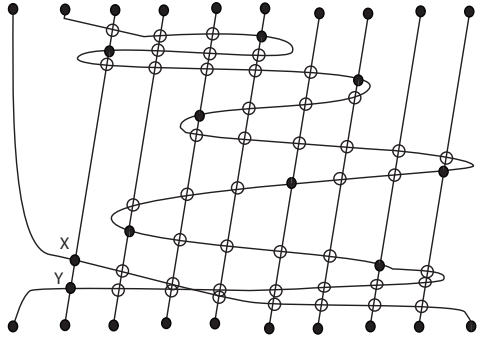}\caption{The
Transformed Braid $\beta'$} \label{brunn}
\end{figure}

It is easy to see that the closure $Cl(\beta')$ differs from the
closure $Cl(\beta)$ by a second Reidemeister move (which is not
braid-like!).

The two added crossings $X,Y$  are both even in the Gaussian parity.
When applying the parity bracket, we see that $[\beta']_{p}$ will
split into $3$ components after closing it up: one component will be
trivial, and two other components will have intersections with each
other. Thus, taking into account that $Cl(\beta)$ is irreducible,
odd and has one component, one can easily see that these bracket can
not be related to each other by bare addition/removal of circles.

I am very grateful to the referee for various useful remarks.


\begin{thebibliography}{99}





\bibitem[IMN]{IMN} D.P.Ilyutko, V.O.Manturov, I.M.Nikonov, {\em Parity in Knot Theory and Graph-Links}, CMFD, 41 (2011),  3–163

\bibitem[Ka]{Kamada} S.Kamada, Braid Presentation of Virtual Knots
and Welded Knots, {\em Osaka J.Math.}, 2007, 44 (2), 441--458.

\bibitem[KM]{KM} L.H.Kauffman, V.O.Manturov, A graphical construction of the $sl_{3}$ invariant for virtual
knots, {\em Quantum Topology}, 5, 2014, p. 1-17.


\bibitem[LR]{LR} S. Lambropoulou and C. P. Rourke, Markov's theorem in
3-manifolds. Special issue on braid groups and related topics
(Jerusalem, 1995), Topology Appl. 78(1–2) (1997) 95–122.


\bibitem[Ma1]{Parity} V.O.Manturov, {\em Parity in Knot Theory}, Mat.
Sbornik, 201:5 (210), pp. 65-110.

\bibitem[MW]{MW} V.O.Manturov, H.Wang, {\em Markov Theorem for Free
Links},  Journal of Knot Theory and Its Ramifications Vol. 21, No.
13 (2012) 1240010 (23 pages)



\end{thebibliography}
\end{document}